\documentclass[10pt,twoside]{article}

\usepackage{7OSME-style}

\theoremstyle{definition}

\newtheorem{prop}{Proposition}
\newtheorem{defn}{Definition}

\newtheorem{cor}{Corollary}[prop]
\newtheorem{rem}{Remark}

\title{Approximating a Target Surface with 1-DOF Rigid Origami}
\authorlist[He, Guest]{Zeyuan He, Simon D. Guest} 
\affiliations{Zeyuan He \at Department of Engineering, University of Cambridge, Trumpington Street, Cambridge CB2 1PZ, United Kingdom, \email{zh299@cam.ac.uk} \and Simon D. Guest \at Department of Engineering, University of Cambridge, Trumpington Street, Cambridge CB2 1PZ, United Kingdom, \email{sdg@eng.cam.ac.uk}}

\makeatletter
  \let\runtitle\@title
  \let\runauthor\shortauthor
\makeatother

\begin{document}

\maketitle

\begin{abstract}
 	We develop some design examples for approximating a target surface at the final rigidly folded state of a developable quadrilateral creased paper, which is folded with a 1-DOF rigid folding motion from the planar state. The final rigidly folded state is reached due to the clashing of panels. Now we can approximate some specific types of non-developable surfaces, but we do not yet fully understand how to approximate an arbitrary surface with a developable creased paper that has limited DOFs. Our designs might have applications in areas related to the formation of a shell structure from a planar region.
\end{abstract}

\section{Introduction}
We discuss here the inverse problem of rigid origami, that is, to approximate a target surface by rigid origami --- usually starting from a planar creased paper. This problem has been preliminarily discussed  in the review \cite{callens_flat_2017}. Generally, we need to consider the following factors when designing a creased paper for approximation.

\begin{enumerate}
	\item Which surface we can approximate.
	\item The DOF (degree of freedom) during the rigid folding motion.
	\item The utilization of materials.
\end{enumerate}
Experience shows that, it is hard to make the creased paper behave well in all aspects. Generally, as we reduce the possible DOFs during the rigid folding motion, less surfaces can be approximated. For example, \cite{demaine_origamizer:_2017} give a universal algorithm to fold a planar creased paper to any piecewise-polygon orientable 2-manifold, which can be used to approximate any orientable 2-manifold. Although to archive the ``watertight'' property, they add small additional features at the vertices and along the edges, this algorithm is practical and guarantees a minimum number of seams. This result is undoubtedly successful, but will have many DOFs during its rigid folding motion, and a large proportion of materials are used in the connections and ``walls'' hidden in the ``tuck'' side. On the other hand, \cite{dudte_programming_2016} and \cite{song_design_2017} use a flat-foldable creased paper. These algorithms guarantee one degree-of-freedom (1-DOF) during the rigid folding motion and high utilization of materials, but only possible for a cylinderical developable surface or a surface of revolution, otherwise the creased paper will not be rigid-foldable. If we triangulate this creased paper, we can approximate more surfaces but without being 1-DOF.

In this article we will focus on a branch of the inverse problem, to design a 1-DOF rigid origami approximating some desired shapes.  More formally, given a connected surface $S$ in $\mathbb{R}^3$, for any positive real number $\epsilon>0$, find a creased paper $(P,C)$, which is the union of a connected planar paper $P \subset \mathbb{R}^2$ and a straight-line crease pattern $C$ embedded on $P$, such that

\begin{enumerate} 
	\item $(P,C)$ is rigid-foldable to its final rigidly folded state $(P',C')$, where the rigid folding motion halts because some panels clash.
	\item $(P,C)$ has one degree of freedom during the rigid folding motion.
	\item the Hausdorff distance $d$ between $S$ and $P$ satisfies $d \le \epsilon$. 
\end{enumerate}
More details of the terminologies used here are given in \cite{he_rigid_2019}. 

By using a family of developable quadrilateral creased papers that are rigid-foldable but not necessarily flat-foldable, we are able to give solutions for approximating some surfaces at the final rigidly folded state, but they cannot be completely arbitrary.

Furthermore, the approximation problem naturally induces an optimization problem, that is, find the ``best'' creased paper that fits the extra presupposed requirements. We will discuss it at the end of the article.

\section{Choosing Design Example}

To our best knowledge, the constraints on 1-DOF rigid-foldability for a creased paper restrict the configurations it may have. Therefore it seems hard to approximate an arbitrary connected surface with 1-DOF rigid origami. Our idea is, from several types of 1-DOF developable and rigid-foldable quadrilateral creased papers we have known \cite{he_new_2018, he_rigid_2020}, we choose some of them as the \textit{design examples}, and study the rigid folding motions of them. For each design example, the profile of its inner vertices can approximate certain types of surfaces. The more design examples we can use, the more surfaces we can approximate.

In this section we will analyze two design examples mentioned in \cite{he_rigid_2020}. The first one is the developable case of the "parallel repeating" type, which is generated among rows of parallel inner creases (Figure \ref{fig: unit columns}); and the second one is the developable case of the "orthodiagonal" type, which is generated among several parallel straight line segments (Figure \ref{fig: orthodiagonal creased paper}).

\begin{figure}[!tb]
	\noindent \begin{centering}
		\includegraphics[width=1\linewidth]{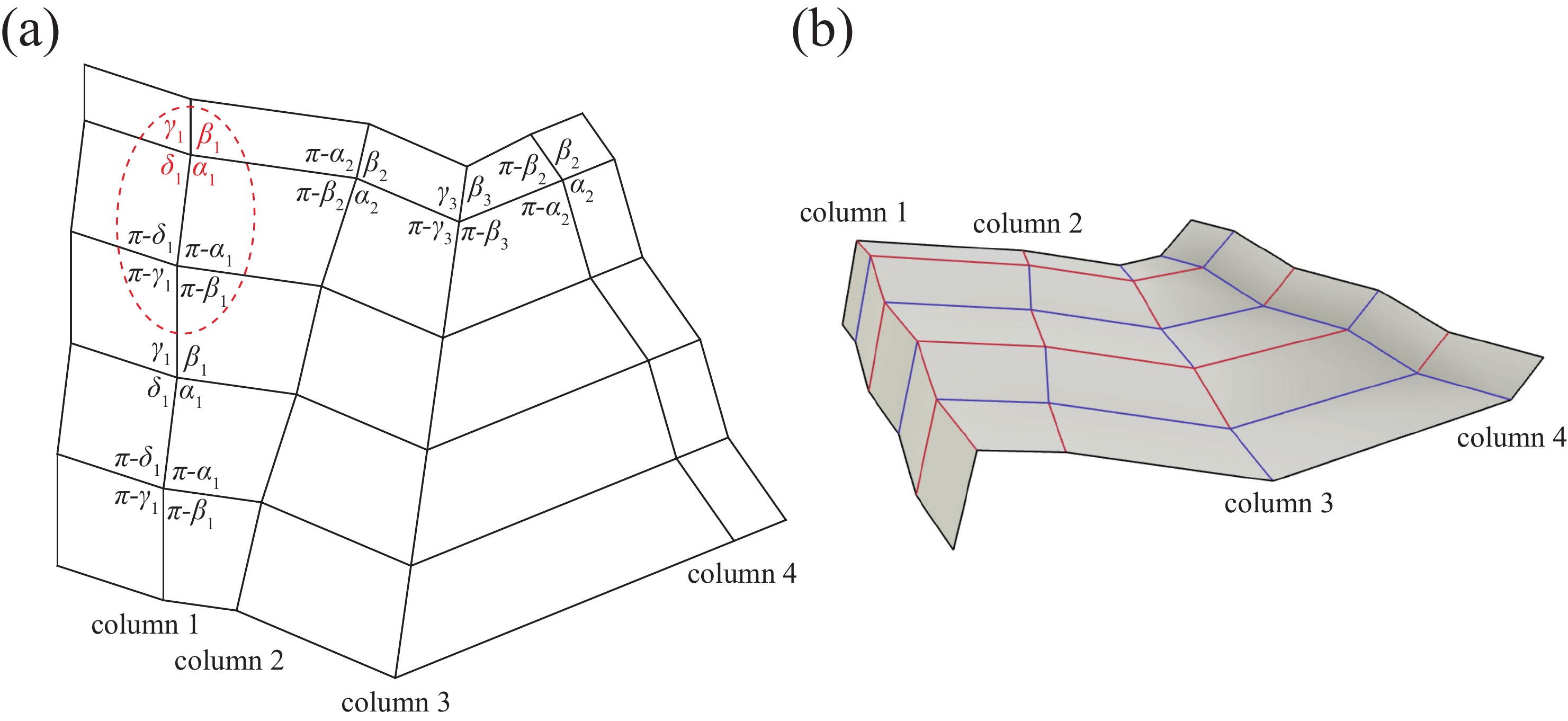}
		\par\end{centering}
	
	\caption{\label{fig: unit columns}(a) is the developable case of the parallel repeating type. In each row there are three independent sector angles $\alpha_i, \beta_i, \gamma_i$, and $\delta_i=2\pi-\alpha_i-\beta_i-\gamma_i$. Column 1 is an example of general input sector angles, where a "basic unit" of this column is labelled by a dashed red cycle; columns 2, 3 and 4 are flat-foldable and ``straight-line'' input sector angles. (b) is a rigidly folded state of (a) plotted by Freeform Origami \cite{tachi_freeform_2010-1}. The mountain and valley creases are colored red and blue. The inner vertices of a column are co-planar (see Proposition \ref{prop co-planar}).}
\end{figure}

\subsection{Parallel Repeating Type}

Figure \ref{fig: unit columns} shows the developable case of the parallel repeating type of rigid-foldable quadrilateral creased papers, which is generated among rows of parallel inner creases. In each column we can choose independent input sector angles $\alpha_i, \beta_i, \gamma_i$, and $\delta_i=2\pi-\alpha_i-\beta_i-\gamma_i$, then construct the rest of the creased paper with both these angles and the supplement of these angles. $\alpha_i, \beta_i, \gamma_i,\delta_i$ should not form a cross. Here the length of creases does not affect the rigid-foldability, but will affect the profile of inner vertices. Based on that we start to analyze which surface the parallel repeating type can approximate.

\begin{prop} \label{prop co-planar}
	The inner vertices on a column of the parallel repeating type are co-planar.
\end{prop} 

\begin{proof}
	Figure \ref{fig: degree-4 single-vertex creased paper}(a) shows a general column. We know $A_1A_2 \parallel A_3A_4$ at any rigidly folded state. Thus $A_1, A_2, A_3, A_4$ are coplanar. Because $A_iA_{i+1} \parallel A_{i+2}A_{i+3}$, this argument continues down the column, which means all the inner vertices are co-planar, and the angle between any two adjacent inner creases is $\xi$, as illustrated in Figure \ref{fig: degree-4 single-vertex creased paper}(b). Figure \ref{fig: degree-4 single-vertex creased paper}(c) demonstrates a rigidly folded state of this column.
\end{proof}

\begin{figure}[!tb]
	\noindent \begin{centering}
		\includegraphics[width=1\linewidth]{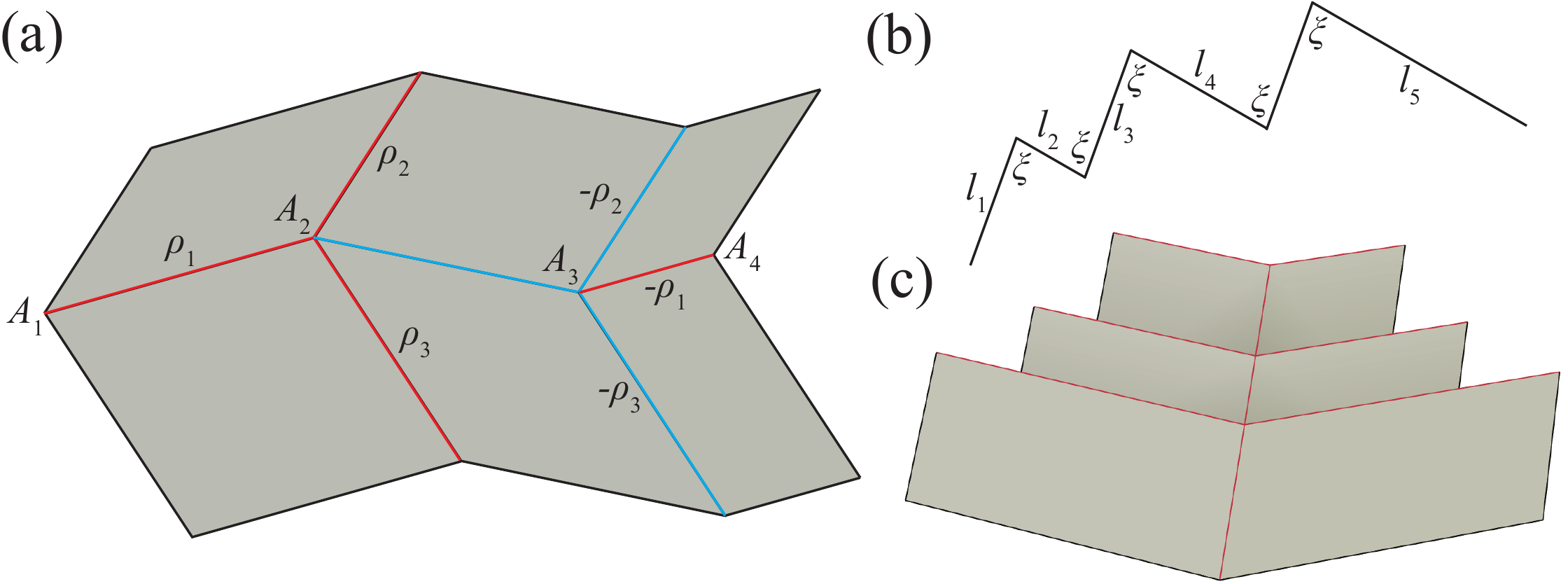}
		\par\end{centering}
	
	\caption{\label{fig: degree-4 single-vertex creased paper}(a) A ``basic unit'' of a column of the parallel repeating type. The mountain and valley creases are colored red and blue. Note that the folding angles $\rho_1, \rho_2, \rho_3$ on corresponding inner creases are opposite. (b) The inner creases of a column with its independent parameters labelled. (c) A view of a column from a view point in the plane coincident with the inner vertices, plotted by Freeform Origami \cite{tachi_freeform_2010-1}.}
\end{figure}

To study the profile of inner vertices on a column of the parallel repeating type, we use

\begin{defn}
	A $x$-$y$ coordinate system is built on the plane mentioned in Proposition \ref{prop co-planar}. We say the inner vertices $(x_i,y_i)$ of a column can \textit{approximate} a given planar curve $f: I \rightarrow \mathbb{R}^2$ if $\forall \epsilon>0$, there exists a column whose inner vertices $(x_i,y_i)$ make the Hausdorff distance $d$ between the set $(x_i,y_i)$ and the curve $f$ satisfy $d \le \epsilon$. 
\end{defn}

\begin{prop} \label{prop: approximation}
	The inner vertices on a column of the parallel repeating type can only approximate a planar curve $f:I \ni t \rightarrow (x(t),y(t)) \in \mathbb{R}^2$ that satisfies the following condition: there exists a rotation $\theta \in [0, 2\pi)$ and a shear transformation of magnitude $\pi/2-\xi$, $\xi \in [0, \pi]$, s.t.\ after the affine transformation $f \rightarrow \overline{f}$ described below, $\overline{f}$ is monotone decreasing.
		\begin{equation} \label{eq: affine}
		\left[ \begin{array}{c}
		\overline{x}(t)\\
		\overline{y}(t)
		\end{array} \right]=
		\left[ \begin{array}{cc}
		1 & -1 / \tan \xi \\
		0 & 1 / \sin \xi
		\end{array} \right]
		\left[ \begin{array}{cc}
		\cos \theta & \sin \theta \\
		-\sin \theta & \cos \theta
		\end{array} \right]
		\left[ \begin{array}{c}
		x(t)\\
		y(t)
		\end{array} \right]
		\end{equation}
We name the curve approximated by the inner vertices of a column as a \textit{target curve}.
\end{prop}

\begin{proof}
Sufficiency: An example of approximation is shown in Figure \ref{fig: target curve}. For a given curve $f$, after a rotation by $\theta$ and a shear transformation by $\pi/2-\xi$, as in equation \eqref{eq: affine}, $f$ (black curve in (b)) is mapped to $\overline{f}$ (black curve in (c)), where the inner creases are parallel to the $\overline{x}$ and $\overline{y}$ axes. Then if $\overline{f}$ is monotonous, we can construct an approximation corresponding to the partition when we apply the Darboux sum to describe the Darboux-integrability (red line segments in (c)). Because $\overline{f}$ is monotonous, it is Darboux-integrable (if $\overline{f}(I)$ is unbounded, the limit of difference between lower and upper Darboux sum is zero when the partition is infinitesimally refined) and only has countable first-kind discontinuity points, so arbitrarily refining the partition will make the Hausdorff distance be arbitrarily small. Hence we can approximate $f$ by re-transforming the approximation in the $\overline{x}$-$\overline{y}$ coordinate system to the $x$-$y$ coordinate system (red line segments in (b)). Here the requirement of monotone decreasing makes the angle between adjacent inner creases $\xi$, not $\pi-\xi$.

Necessity: If a curve $f$ can be approximated by the inner vertices of a column, we can always find corresponding $\theta$ and $\xi$ and do the affine transformation in equation \eqref{eq: affine}. Because the approximation turns left and right alternatively, $\overline{f}$ must be monotonous. To make the angle between adjacent inner creases $\xi$, not $\pi-\xi$, $\overline{f}$ should be monotone decreasing. 
\end{proof}

\begin{figure}[!tb]
	\noindent \begin{centering}
		\includegraphics[width=1\linewidth]{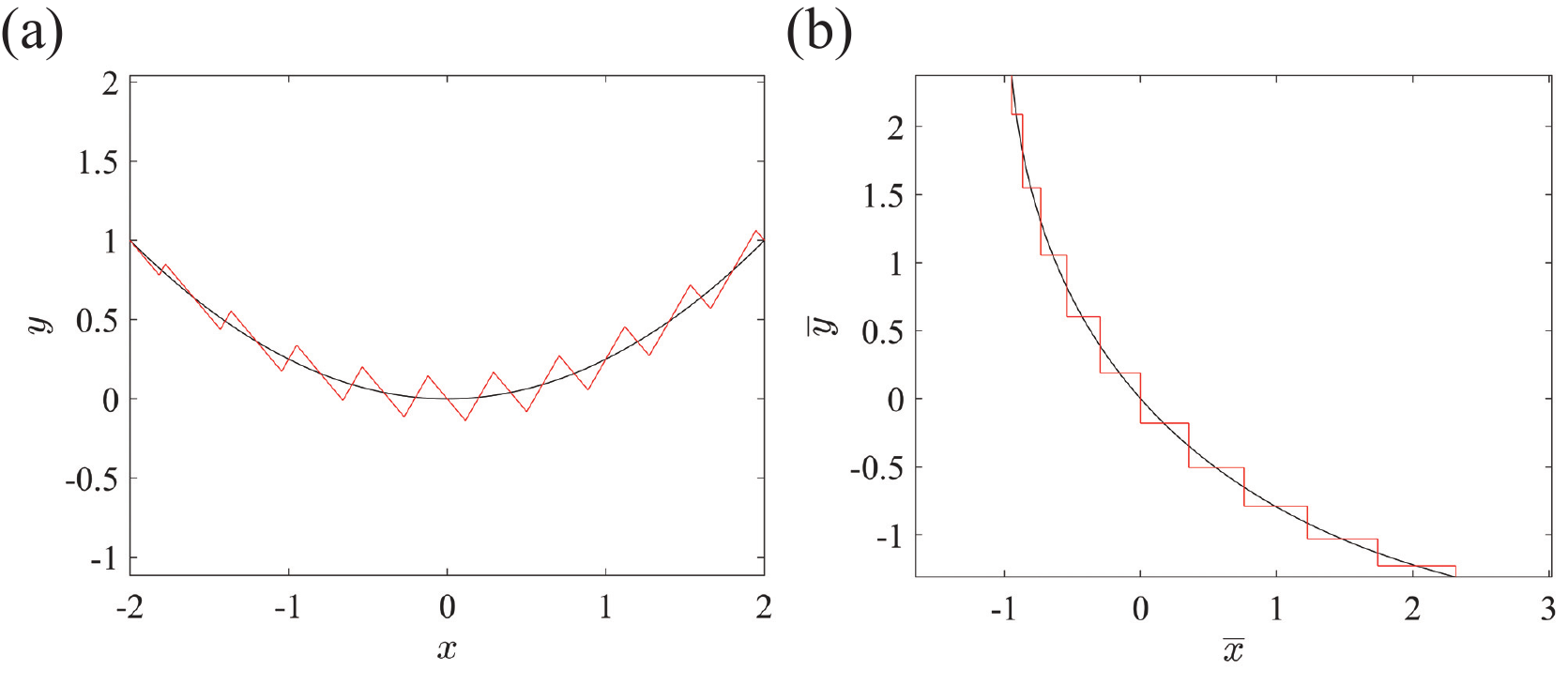}
		\par\end{centering}
	
	\caption{\label{fig: target curve}(a) shows a target curve (coloured black) and the result of an approximation (coloured red). Here $f(t)=[t,t^2/4]^T$, $t \in [-2,2]$. (b) is the image of (a) under equation \eqref{eq: affine} with $\theta=70^{\circ}$ and $\xi=60^{\circ}$. The image of the target curve and the approximation are coloured black and red respectively.}
\end{figure}

\begin{cor}
	The inner vertices of a column cannot approximate a closed planar curve.
\end{cor}

\begin{rem} \label{rem: comment}
	We can require $f$ to be continuous and $I$ to be a closed interval in $\mathbb{R}$ for the following reason. If $f(I)$ is not connected, $f(I)$ is the union of countable disjoint closed intervals. Each such subset can be approximated independently because the rigid folding motions of disconnected creased papers are independent. Therefore we can require $f(I)$ to be connected, which means $\overline{f}(I)$ is continuous and $f(I)$ is continuous. Then we can re-parametrize $f(I)=f(I')$, where $I'$ is a closed interval in $\mathbb{R}$. 
\end{rem}

Next we will analyze the rigid folding motion of a row of the parallel repeating type. 

\begin{prop} \label{prop datum curve}
	The inner vertices on a row of the parallel repeating type can approximate a Darboux-integrable curve $\Gamma: J \rightarrow \mathbb{R}^3$ at its final rigidly folded state, named the \textit{datum curve}. 
\end{prop}

\begin{figure}[!tb]
	\noindent \begin{centering}
		\includegraphics[width=1\linewidth]{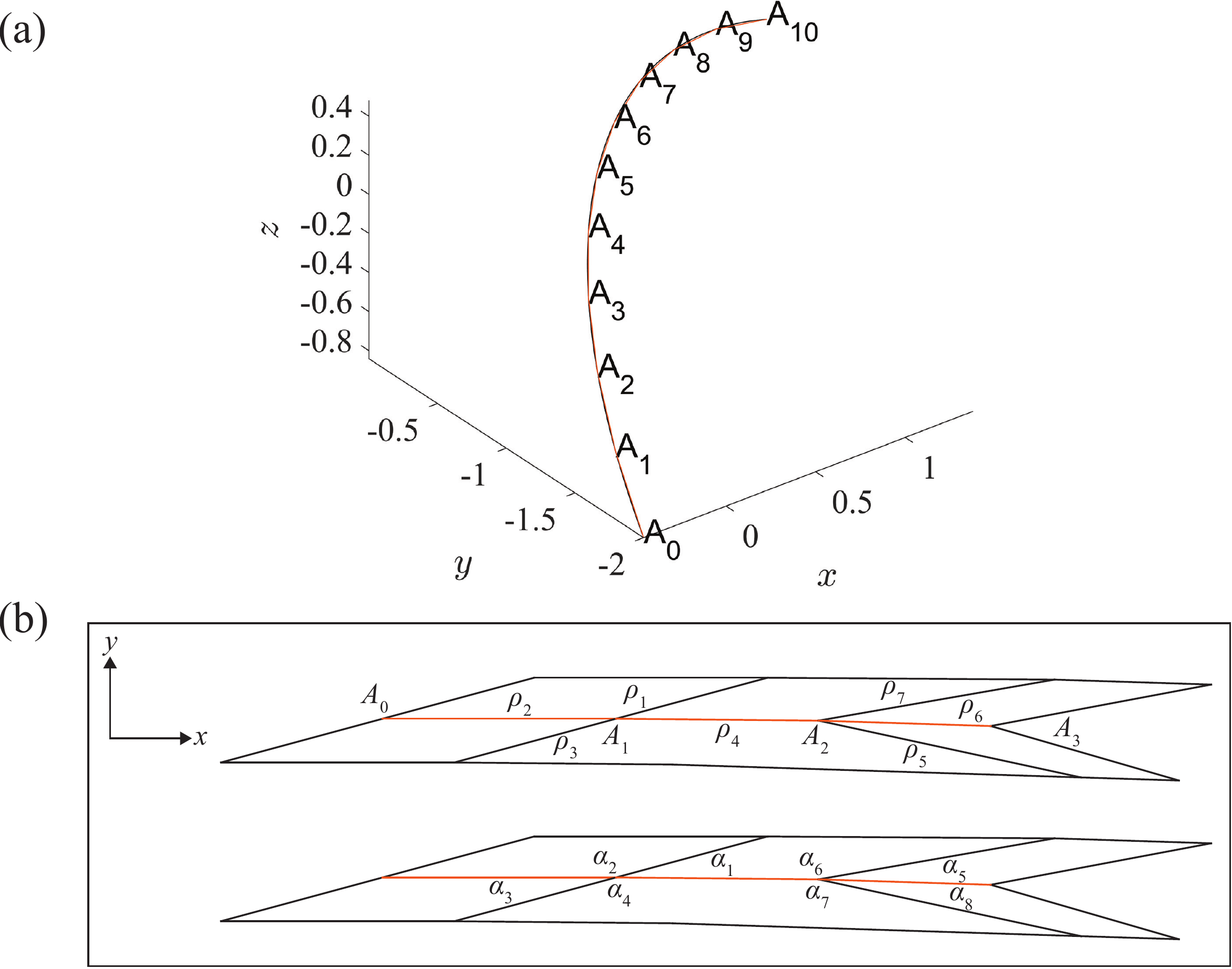}
		\par\end{centering}
	
	\caption{\label{fig: approximation}(a) An example of the datum curve $\Gamma$ (coloured black) with its approximation by line segments (coloured orange). In this example, $n=9$, $\Gamma(u)=[u+\cos u, -2u^2, \sin u]^T$, $u \in [-1,0.5]$, $\rho_4=5\pi/6$. Here we partition $\Gamma$ uniformly in $u$. (b) Part of the creased paper. The inner creases of its final rigidly folded state (coloured orange) exactly form the approximation in (a). We label all partition points $A_k$, $k \in [0,n+1]$; folding angles $\rho_k$, $k \in [1,3n+1]$; sector angles $\alpha_k$, $k \in [1,4n]$. The planar pattern is viewed along the direction perpendicular to the plane.}
\end{figure}

\begin{proof}
	An approximation can be generated following the steps mentioned below.
\begin{enumerate}
	\item Partition $\Gamma$ (black curve in Figure \ref{fig: approximation}(a)) by $n+2$ points and connect adjacent partition points in sequence by line segments (orange line segments in Figure \ref{fig: approximation}(a)). Any given Hausdorff distance $\epsilon$ can be satisfied by choosing a sufficiently large $n$. 
	\item Assign a direction of $\Gamma$, label each partition point $A_i$ ($i \in [0,n+1]$) along this direction. From the coordinates of $A_i$, calculate $l_i= \|A_iA_{i+1}\|$ ($i \in [0,n]$), $\beta_i=\angle A_{i-1}A_iA_{i+1}$ ($i \in [1,n]$) and $\theta_i=\langle \triangle A_{i-1}A_iA_{i+1}, \triangle A_iA_{i+1}A_{i+2} \rangle$ ($i \in [1,n-1]$). $\theta_i \in [0,2\pi)$ is calculated by the rotation angle along the vector $\overrightarrow{A_iA_{i+1}}$.
	\item Choose the first inner vertex on the left from column 3 in Figure \ref{fig: unit columns}, set $\rho_4$ as a specific angle and $\rho_2=\pi$ (see Figure \ref{fig: approximation}(b)), with a known $\beta_1$, solve the following equations:
	\begin{equation} \label{eq: the first beta and theta}
	\begin{gathered} 
	\rho_2=2\arccos\bigg(\dfrac{\cos \alpha_2 \cos \beta_1-\cos \alpha_1}{\sin \alpha_2 \sin \beta_1}\bigg) \\	
	\rho_4=2\arccos\bigg(\dfrac{\cos \alpha_1 \cos \beta_1-\cos \alpha_2}{\sin \alpha_1 \sin \beta_1}\bigg) 			
	\end{gathered}
	\end{equation}	
	we can obtain $\alpha_1$ and $\alpha_2$.
    \item Then we continue to obtain other sector angles ($i \in [1,n-1]$) in this row, as shown in Figure \ref{fig: approximation}(b). Regarding $\alpha_{4i-3}$, $\alpha_{4i-2}$, $\alpha_{4i-1}$, $\alpha_{4i}$, $\beta_i$, $\beta_{i+1}$ and $\theta_i$ as known variables, there are four equations for $\alpha_{4i+1}$, $\alpha_{4i+2}$, $\alpha_{4i+3}$ and $\alpha_{4i+4}$. Two of them are related to $\beta_i$, $\beta_{i+1}$ and $\theta_i$.
    \begin{multline} \label{eq: general beta and theta}
    \arccos\bigg(\dfrac{\cos \alpha_{4i-3} \cos \beta_i-\cos \alpha_{4i-2}}{\sin \alpha_{4i-3} \sin \beta_i}\bigg) \pm \arccos\bigg(\dfrac{\cos \alpha_{4i-1}-\cos \alpha_{4i} \cos \beta_i}{\sin \alpha_{4i} \sin \beta_i}\bigg) = \\
    \arccos\bigg(\dfrac{\cos \alpha_{4i+2} \cos \beta_{i+1}-\cos \alpha_{4i+1}}{\sin \alpha_{4i+2} \sin \beta_{i+1}}\bigg) \pm \arccos\bigg(\dfrac{\cos \alpha_{4i+4}-\cos \alpha_{4i+3} \cos \beta_{i+1}}{\sin \alpha_{4i+3} \sin \beta_{i+1}}\bigg) \\
    \shoveleft{\theta_i= \pm \arccos\bigg(\dfrac{\cos \alpha_{4i-2}-\cos \alpha_{4i-3} \cos \beta_i}{\sin \alpha_{4i-3} \sin \beta_i}\bigg)} \\ 
    \pm \arccos\bigg(\dfrac{\cos \alpha_{4i+1}-\cos \alpha_{4i+2} \cos \beta_{i+1}}{\sin \alpha_{4i+2} \sin \beta_{i+1}}\bigg)
    \end{multline}
    Note that the $\pm$ depends on the rigid folding motion we choose and the magnitude of sector angles. These equations can be directly derived from spherical trigonometry. Another equation is $\alpha_{4i+1}+\alpha_{4i+2}+\alpha_{4i+3}+\alpha_{4i+4}=2\pi$, but in order to simplify the steps we suppose the other vertices in this row are from column 2 in Figure \ref{fig: unit columns}. Hence there are two more equations:
	\begin{equation}
	\alpha_{4i+1}+\alpha_{4i+3}=\pi \quad \alpha_{4i+2}+\alpha_{4i+4}=\pi
	\end{equation}
    \item With $l_i$ and $\alpha_{4i-3}$, $\alpha_{4i-2}$, $\alpha_{4i-1}$, $\alpha_{4i}$ ($i \in [1,n]$), draw the creased paper.
\end{enumerate}
\end{proof}

\begin{rem} \label{rem: comment 2}
	Generically, We can require $\Gamma$ to be continuous and $J$ to be a closed interval in $\mathbb{R}$ for the following reason. Similar to our analysis in Remark \ref{rem: comment}, $\Gamma(J)$ can be required as a connected and closed set. If $\Gamma(J)$ has no second-kind discontinuity points, $\Gamma(J)$ is continuous. Then we can re-parametrize $\Gamma(J)=\Gamma(J')$, where $J'$ is a closed interval in $\mathbb{R}$. Besides, $\Gamma$ can be a closed curve.
\end{rem}

\begin{rem}
	Here the final rigidly folded state is not special, we make it final by designing the rigid folding motion to be halted by clashing of panels at the first column of inner vertices from the left. If we release this condition we can make the datum curve be approximated by an intermediate rigidly folded state.
\end{rem}

\begin{cor} \label{cor: s1-s1}
	In step 4 of Proposition \ref{prop datum curve}, If we choose other vertices in the first row from column 3 in Figure \ref{fig: unit columns} , equation \eqref{eq: general beta and theta} will be simplified:
	\begin{equation}
	\begin{gathered}
	\dfrac{\cos \alpha_{4i-3} \cos \beta_i-\cos \alpha_{4i-2}}{\sin \alpha_{4i-3} \sin \beta_i} 
	=\dfrac{\cos \alpha_{4i+2} \cos \beta_{i+1}-\cos \alpha_{4i+1}}{\sin \alpha_{4i+2} \sin \beta_{i+1}} \\
	\theta_1 = 0 ~~ \textrm{if} ~~ (\alpha_{4i-3}+\alpha_{4i-2}-\pi)(\alpha_{4i+1}+\alpha_{4i+2}-\pi)>0 \\
	~~~~ = \pi ~~ \textrm{if} ~~ (\alpha_{4i-3}+\alpha_{4i-2}-\pi)(\alpha_{4i+1}+\alpha_{4i+2}-\pi)<0
	\end{gathered}
	\end{equation}
	Now there is only one branch of rigid folding motion and only one equation for $\alpha_{4i+1}$ and $\alpha_{4i+2}$. However, it requires $\theta_1=0$ or $\pi$, which means $\Gamma$ should be locally planar in a discrete sense. More essentially, \cite{fuchs_more_1999} illustrate this in a continuous sense.
\end{cor}

From Propositions \ref{prop co-planar}, \ref{prop: approximation} and \ref{prop datum curve}, we are now in a position of studying which surface the parallel repeating type can approximate.

\begin{figure}[!tb]
	\noindent \begin{centering}
		\includegraphics[width=0.98\linewidth]{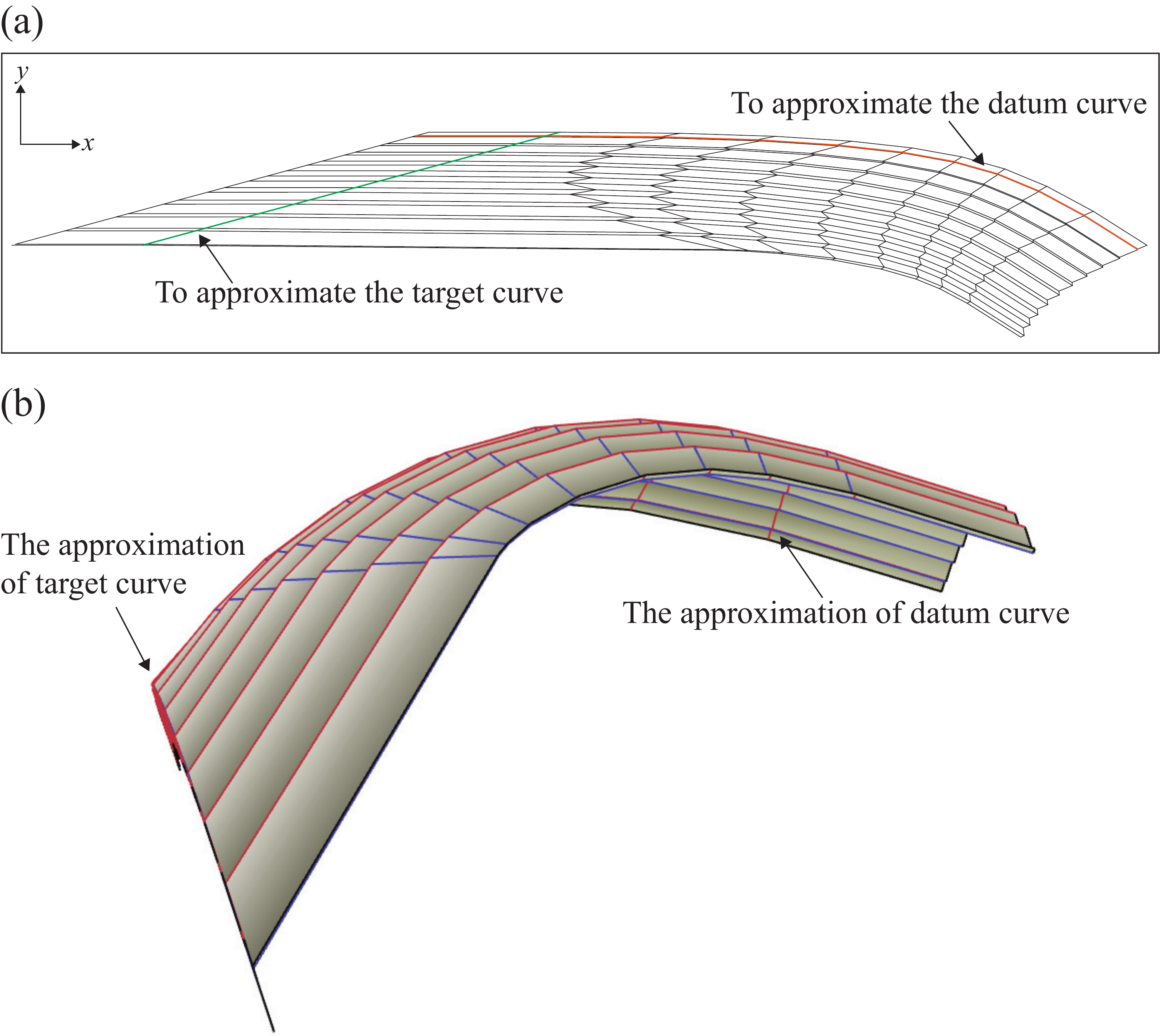}
		\par\end{centering}
	
	\caption{\label{fig: approximation-2}(a) is the creased paper approximating the target surface described in Proposition \ref{prop: surface}. The equation of the datum curve is given in the caption of Figure \ref{fig: approximation}; $f_1(t)=[t,\exp(t)]$, $t \in [0,1]$. We choose $\theta=73^\circ$. The orange and green inner creases are to approximate the datum curve $\Gamma$ and the target curve $f_1$. The first column of inner vertices from the left approximates the datum curve is from column 3 in Figure \ref{fig: unit columns}, and the others are from column 2 in Figure \ref{fig: unit columns}. The planar creased paper is viewed along the direction perpendicular to the plane. (b) is the final rigidly folded state of (a), where we approximate these two curves, plotted by Freeform Origami \cite{tachi_freeform_2010-1}. The rigid folding motion halts due to a clash at the column we approximate the target curve. The mountain and valley creases are colored red and blue.}
\end{figure}

\begin{prop} \label{prop: surface}
	The parallel repeating type can approximate a surface $S: V \ni (u,t) \rightarrow \vec{r}(u,t) \in \mathbb{R}^3$ under a given Hausdorff distance $\epsilon$. The creased paper is generated by the following steps, and $S$ is described below. (see Figure \ref{fig: approximation-2})
	\begin{enumerate}
		\item Choose a datum curve $\Gamma: u \in J \rightarrow \mathbb{R}^3$ and approximate it in a row under a sufficiently small Hausdorff distance $\epsilon_1<\epsilon$, as described in Proposition \ref{prop datum curve}.
		\item Choose a target curve $f_1: t \in I \rightarrow \mathbb{R}^2$ and approximate it the first column from the left under a sufficiently small Hausdorff distance $\epsilon_2<\epsilon$, as described in Proposition \ref{prop: approximation}.
		\item In other columns, the shape of the target curve $f_{i+1} : I \ni t \rightarrow (x_{i+1}(t),y_{i+1}(t))$ ($i \in [1,n-1]$) is an affine transformation of $f_1 : I \ni t \rightarrow (x_1(t),y_1(t))$.
		\begin{equation} \label{eq: affine 2}
		\left[ \begin{array}{c}
		x_{i+1}(t)\\
		y_{i+1}(t)
		\end{array} \right]= A^{-1}(\xi_{i+1},\theta)
		\prod_{j=1}^{i} \left[ \begin{array}{cc}
		k_1 & 0 \\
		0 & k_2
		\end{array} \right] A(\xi_1,\theta)
		\left[ \begin{array}{c}
		x_1(t)\\
		y_1(t)
		\end{array} \right] \\
		\end{equation}
		where,
		\begin{equation*}
		k_1= \frac{\sin \alpha_{4j-3}}{\sin \alpha_{4j+2}}, \quad
		k_2= \frac{\sin \alpha_{4j}}{\sin \alpha_{4j+3}}
		\end{equation*}
		or vice versa. This depends on whether the line segments start in the $\overline{x}$ or $\overline{y}$ direction as shown in Figure \ref{fig: target curve}(b). Additionally,
		\begin{equation*}
		\begin{gathered}
		A(\xi,\theta)=
		\left[ \begin{array}{cc}
		1 & -1 / \tan \xi \\
		0 & 1 / \sin \xi
		\end{array} \right]
		\left[ \begin{array}{cc}
		\cos \theta & \sin \theta \\
		-\sin \theta & \cos \theta
		\end{array} \right] \\
		\cos \xi_{i+1}=\cos \alpha_{4i+2}\cos \alpha_{4i+3}+\dfrac{\sin \alpha_{4i+2}\sin \alpha_{4i+3}}{\sin \alpha_{4i-3}\sin \alpha_{4i}}(\cos \xi_i - \cos \alpha_{4i-3}\cos \alpha_{4i}) \\
		\xi_1=|2\alpha_2-\pi|
		\end{gathered}
		\end{equation*}
		$A$ is the affine matrix we used in Proposition \ref{prop: approximation}, and $\theta$ is the rotation angle in the approximation of $f_1$.
		We then define the angle between the planes where $f_i$ and $f_{i+1}$ locate by $\phi_i$ ($i \in [1,n-1]$), which can be expressed by
		\begin{equation} \label{eq: phi}
		\begin{gathered}
		\eta_1=\arccos\bigg(\dfrac{\cos \alpha_{4i}-\cos \alpha_{4i-3} \cos \xi_i}{\sin \alpha_{4i-3} \sin \xi_i}\bigg) \\
		\eta_2=\arccos\bigg(\dfrac{\cos \alpha_{4i+3}-\cos \alpha_{4i+2} \cos \xi_{i+1}}{\sin \alpha_{4i+2} \sin \xi_{i+1}}\bigg) \\
		\cos \phi_i=- \cos \eta_1 \cos \eta_2-\sin \eta_1\sin \eta_2 \cos (\alpha_{4i-3}+\alpha_{4i+2})
		\end{gathered}
		\end{equation}
	\end{enumerate}
	When the Hausdorff distance $\epsilon \rightarrow 0$, the surface $S$ can be expressed as:
	\begin{equation} \label{eq: surface}
	\vec{r}(u,t)=\Gamma(u)+\widetilde{f}_u(t)
	\end{equation}
	where $\widetilde{f}_u$ depends on $u$ (equation \eqref{eq: affine 2}) and locates on different planes determined by $\phi_i$ (equation \eqref{eq: phi}).
\end{prop}

\begin{cor} \label{cor first special}
	If $\Gamma$ is piecewise-spiral, on each piece $\Gamma_j$, $\alpha_{4i+1}$, $\alpha_{4i+2}$, $\alpha_{4i+3}$ and $\alpha_{4i+4}$ ($i \in [1,n-1]$) will keep the same from the second column respectively. Then each piece $S_j$ induced by $\Gamma_j$ can be expressed analytically as $f_1$ scanning along $\Gamma$. Some examples are demonstrated in \cite{song_design_2017} when $\Gamma$ is circular. Besides, if $f_1$ is piecewise-linear, $f_u$ remains piecewise-linear.
\end{cor}

\subsection{Orthodiagonal Type}

The developable case of the orthodiagonal type of rigid-foldable quadrilateral creased papers is generated among several parallel line segments \cite{he_rigid_2020}, as shown in Figure \ref{fig: orthodiagonal creased paper}. For all $i,j$, The sector angles here should satisfy
\begin{equation} \label{eq: non-stitching}
\dfrac{\tan \alpha_{ij}}{\tan \alpha_{ij+1}}=\dfrac{\tan \alpha_{i+1j}}{\tan \alpha_{i+1j+1}}
\end{equation}
$i \ge 1$, $j \ge 0$. This equation guarantees each column and each row of inner vertices to be co-planar during the rigid folding motion. 

\begin{figure}[!tb]
	\noindent \begin{centering}
		\includegraphics[width=1\linewidth]{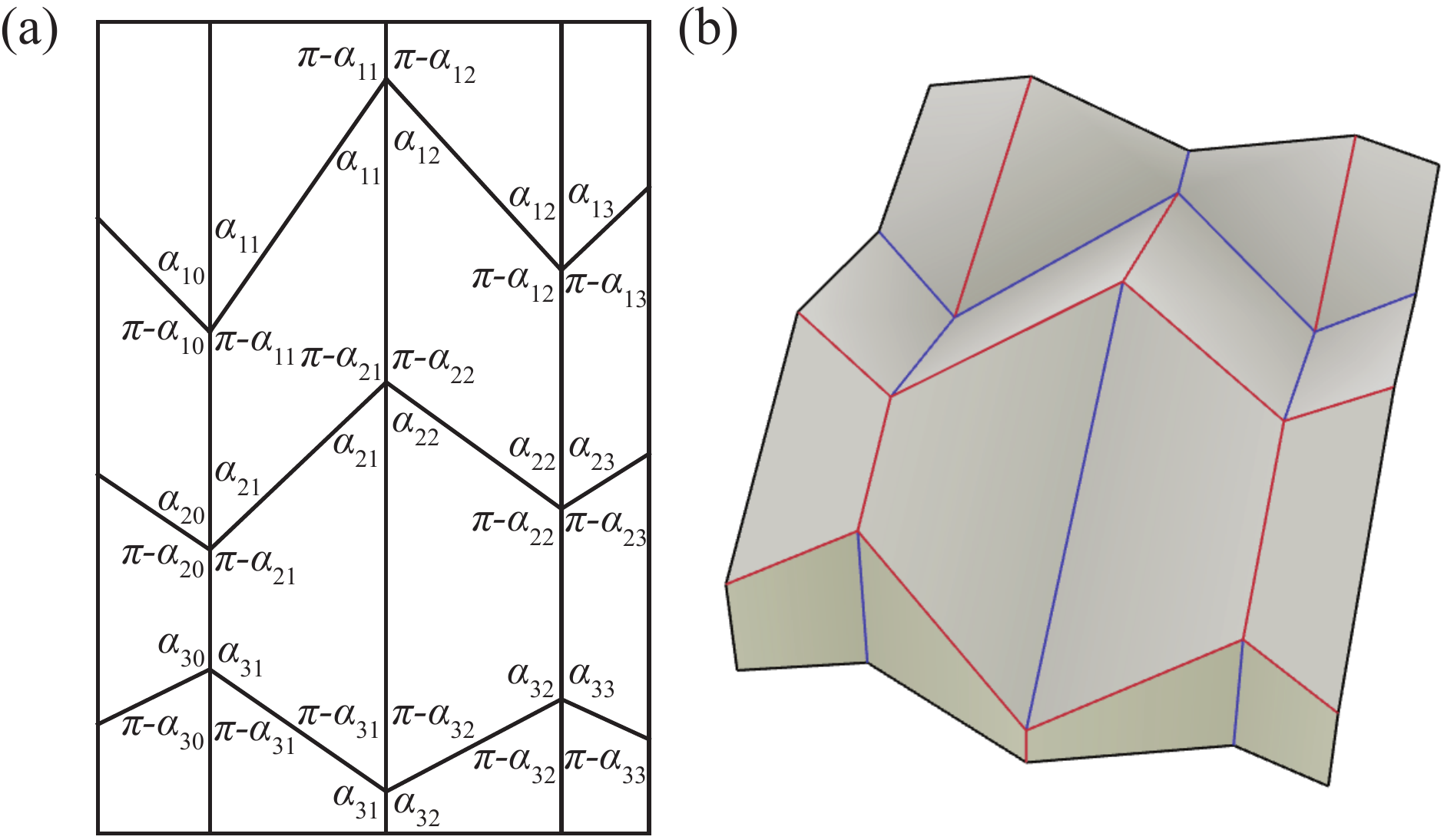}
		\par\end{centering}
	
	\caption{\label{fig: orthodiagonal creased paper}(a) is an example of the developable case of the orthodiagonal type, where we label out the relations among the sector angles. Additionally, the sector angles should satisfy equation \eqref{eq: non-stitching}. (b) is a rigidly folded state of (a), plotted by Freeform Origami \cite{tachi_freeform_2010-1}. The mountain and valley creases are colored red and blue.}
\end{figure}

\begin{prop} \label{prop: datum curve 2}
	The inner vertices on a column of the orthodiagonal type can approximate a Darboux-integrable curve $\Gamma: J \rightarrow \mathbb{R}^2$, named a \textit{datum curve}.
\end{prop}

\begin{proof}
An approximation can be generated following the steps mentioned below.
\begin{enumerate}
	\item Under a given Hausdorff distance $\epsilon$, partition $\Gamma$ (black curve in Figure \ref{fig: straight-line approximation}(a)) by $n+2$ points and connect adjacent partition points in sequence by line segments (orange line segments in Figure \ref{fig: straight-line approximation}(a)). To make the angle between adjacent inner creases not too close to $\pi$ we choose partition points on an $\epsilon$-tube of $\Gamma$ (purple curves in Figure \ref{fig: straight-line approximation}(a)). Note that here we use a method different from step 1 in Proposition \ref{prop datum curve}, and there are many techniques to approximate a Darboux-integrable curve by a series of line segments.
	\item Assign a direction of $\Gamma$, label each partition point $A_i$ ($i \in [0,n+1]$) along this direction. From the coordinates of $A_i$, calculate $l_i= \|A_iA_{i+1}\|$ ($i \in [0,n]$) and $\beta_i=\angle A_{i-1}A_iA_{i+1}$ ($i \in [1,n]$).
	\item Without loss of generality, we just consider the first column from the left. Calculate all the sector angles on the left $\alpha_{i0}$ ($i \in [1,n]$) from equation \eqref{eq: straight-line alpha 1}, which is to make the rigid folding motion halt at the left side of the first column.
	\begin{equation} \label{eq: straight-line alpha 1}
	\alpha_{i0}=\dfrac{\pi \pm \beta_i}{2}
	\end{equation}	
	Note that the $\pm$ depends on whether the line segments (coloured orange in Figure \ref{fig: straight-line approximation}(a)) ``turn left'' or ``turn right'' along the direction assigned for $\Gamma$. Only one of $\alpha_{i1}$ can be randomly chosen. Suppose it is $\alpha_{11}$, which should satisfy equation \eqref{eq: straight-line alpha 2} to make sure the inner creases turn left or right properly, and the rigid folding motion halts at the left side.
	\begin{equation} \label{eq: straight-line alpha 2}
	\begin{gathered}
	(\alpha_{11}-\pi/2)(\alpha_{10}-\pi/2)>0 \\
	0<|\alpha_{11}-\pi/2|<|\alpha_{10}-\pi/2|
	\end{gathered}
	\end{equation}
	the other sector angles $\alpha_{i1}$ ($i \in [2,n]$) can be calculated from equation \eqref{eq: non-stitching}.
	\item With $l_i$ and $\alpha_{i0}$, $\alpha_{i1}$ ($i \in [1,n]$), draw the creased paper.
\end{enumerate}	
\end{proof}

\begin{rem}
	As shown in Remark \ref{rem: comment 2}, generically, we can require $\Gamma$ to be continuous and $J$ to be a closed interval in $\mathbb{R}$. Besides, $\Gamma$ can be a closed curve. Another point is the approximation of $\Gamma$ does not need to turn left or right alternately, which means the mountain-valley assignment in each row not necessarily change alternately.
\end{rem}

Then we will analyze which surface the orthodiagonal type can approximate. From equation \eqref{eq: non-stitching}, we know the sector angles in the first column and first row are independent, therefore we can also approximate a Darboux-integrable curve with the inner vertices in a row. The problem of such approximation is, we cannot write a clear expression for the curve approximated by other rows of inner vertices in this creased paper, which makes it hard to grasp the feature of the target surface. (In step 3 of Proposition \ref{prop: surface}, we express the curve approximated by other columns of inner vertices as an affine transformation of the target curve.) Based on that we set $\alpha_{1j+1}=\alpha_{1j}$ ($j \ge 1$), and the result in Proposition \ref{prop: approximation} can be applied. For the orthodiagonal type, this simplification makes it possible to express the curve approximated by other rows of inner vertices as an affine transformation of the curve approximated by the first row of inner vertices. If we regard the surface approximated by such a simplified creased paper as a piece, the target surface can consist of these pieces stitched in the transverse direction, as shown in Proposition \ref{prop: approximation 2}.

\begin{figure}[!tb]
	\noindent \begin{centering}
		\includegraphics[width=1\linewidth]{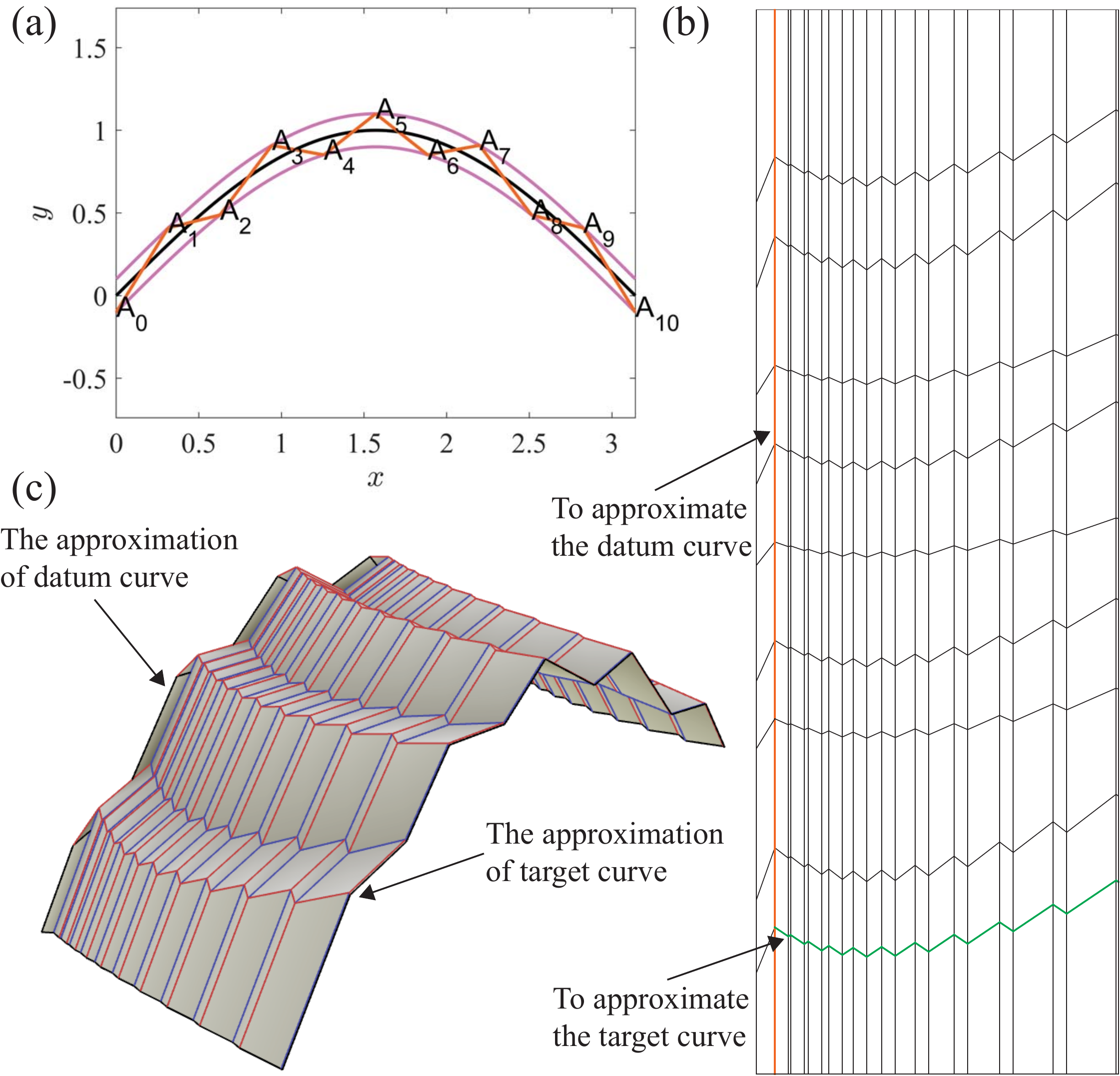}
		\par\end{centering}
	
	\caption{\label{fig: straight-line approximation}(a) An example of the datum curve $\Gamma$ (coloured black) with its approximation by line segments (coloured orange) generated from a $\epsilon$-tube of $\Gamma$ (coloured purple). In this example, $n=9$, $\Gamma(u)=[u, \sin u]^T$, $u \in [0,\pi]$, $\alpha_1=\pi/4+\alpha_2/2$. Here we partition $\Gamma$ uniformly in $u$. (b) is the creased paper approximating a piece described in Proposition \ref{prop: approximation 2}, and $f_1(t)=[t,t-\ln(t)]^T$, $t \in [0.5,1.5]$. We choose $\theta=30^\circ$. The orange and green inner creases are to approximate the datum curve $\Gamma$ and the target curve $f_1$. (c) is the final rigidly folded state of (b), where we approximate these two curves, plotted by Freeform Origami \cite{tachi_freeform_2010-1}. The rigid folding motion halts due to a clash at the column we approximate the datum curve. The mountain and valley creases are colored red and blue.}
\end{figure}

\begin{prop} \label{prop: approximation 2}
	The orthodiagonal type can approximate a surface $S$ which is stitched by countable pieces $S_k: V \ni (u,t) \rightarrow \vec{r}(u,t) \in \mathbb{R}^3$ in the transverse direction under a given Hausdorff distance $\epsilon$. Each piece is generated by the following steps. (see Figure \ref{fig: straight-line approximation})
	\begin{enumerate}
		\item Choose a datum curve $\Gamma: u \in J \rightarrow \mathbb{R}^2$ and approximate it in a column under a sufficiently small Hausdorff distance $\epsilon_1<\epsilon$, as described in Proposition \ref{prop: datum curve 2}.
		\item Choose a target curve $f_1: t \in I \rightarrow \mathbb{R}^2$ and approximate it in the first row from the top under a sufficiently small Hausdorff distance $\epsilon_2<\epsilon$, as described in Proposition \ref{prop: approximation}.
		\item Here $\alpha_{1j+1}=\alpha_{1j}$ ($j \ge 1$). With equation \eqref{eq: non-stitching}, calculate the other sector angles in the creased paper. In other columns, the shape of the target curve $f_i : I \ni t \rightarrow (x_i(t),y_i(t))$ ($i \in [2,n]$) is an affine transformation of $f_1 : I \ni t \rightarrow (x_1(t),y_1(t))$.
		\begin{equation} \label{eq: straight-line target curves}
		\left[ \begin{array}{c}
		x_i(t)\\
		y_i(t)
		\end{array} \right]= \frac{\sin \alpha_{i1}}{\sin \alpha_{11}} A^{-1}(\xi_i,\theta) A(\xi_1,\theta)
		\left[ \begin{array}{c}
		x_1(t)\\
		y_1(t)
		\end{array} \right]
		\end{equation}
		where,
		\begin{equation*}
		A(\xi,\theta)=
		\left[ \begin{array}{cc}
		1 & -1 / \tan \xi \\
		0 & 1 / \sin \xi
		\end{array} \right]
		\left[ \begin{array}{cc}
		\cos \theta & \sin \theta \\
		-\sin \theta & \cos \theta
		\end{array} \right]
		\end{equation*}
		is the affine matrix we used in Proposition \ref{prop: approximation}, and for $i \in [1,n]$
		\begin{equation*}
		\cos \xi_{i}=\dfrac{4 \cos^2\alpha_{i1}}{1-\cos\beta_i}-1
		\end{equation*}
		$\theta$ is the rotation angle in the approximation of $f_1$. $f_i$ locates on a plane perpendicular to $\Gamma$.
	\end{enumerate}
	When the Hausdorff distance $\epsilon \rightarrow 0$, the piece $S_k$ can be expressed as:
	\begin{equation} 
	\vec{r}(u,t)=\Gamma(u)+\widetilde{f}_u(t)
	\end{equation}
	where $\widetilde{f}_u$ depends on $u$ (equation \eqref{eq: straight-line target curves}) and locates on the tangent plane of $\Gamma$ at $u$ that is also perpendicular to $\Gamma$. 
\end{prop}

\begin{cor} \label{cor second special}
	If $\Gamma$ is piecewise-circular, on each piece $\Gamma_j$, $\alpha_{i0}$ and $\alpha_{i1}$ ($i \in [1,n]$) will keep the same respectively. Then each piece $S_j$ induced by $\Gamma_j$ can be expressed analytically as $f_1$ scanning along $\Gamma_j$. Besides, if $f_1$ is piecewise-linear, $f_u$ remains piecewise-linear, then each piece induced by $f_1$ will become a cylinderical developable surface.    
\end{cor}

\section{Discussion}

\subsection{Comment on the Algorithms}

This article presents some initial results for the problem we have set. We recognize there are flaws and limitations in the algorithms, some of which are discussed here.

\begin{enumerate}
	\item In Proposition \ref{prop: approximation}, for a given target curve $f$ and an angle $\xi$, the condition we give to find an appropriate $\theta$ is not convenient. For the examples in this article we try scattered $\theta \in [0,2\pi)$ to find possible approximations.
	\item In Proposition \ref{prop datum curve}, the folding angle $\rho_4$ is a flexible parameter, that can be used to adjust the sector angles $\alpha_1$ and $\alpha_2$. Actually, $\rho_4$ is the magnitude of all the folding angles on every row of inner creases. If we set $\rho_4$ too close to $\pi$, $\alpha_1$ will increase, but the width of approximation will be too small in the longitudinal direction, which may not be suitable for application. If we set $\rho_4$ too far from $\pi$, the singularity of solutions will increase. For the approximation shown in Figures \ref{fig: approximation}(c) and \ref{fig: approximation}(d) we choose $\rho_4=5\pi/6$.
	\item The algorithm in Proposition \ref{prop datum curve} does not guarantee a solution. We need to assume the rigid folding motion and choose the signs that make equation \eqref{eq: general beta and theta} most likely to be solved. If it happens that the parameters are on the singular points of equations \eqref{eq: the first beta and theta} and \eqref{eq: general beta and theta}, there may be no solution.
	\item Even if we have obtained all the sector angles, when we plot the creased paper as described in Propositions \ref{prop: surface} and \ref{prop: approximation 2}, the inner creases in different columns may intersect at points other than vertices. We haven't found a good way to control this. Sometimes scaling the datum or target curve will help.
	\item In Proposition \ref{prop: datum curve 2}, the sector angle $\alpha_1$ is also a flexible parameter, which can be used to control the width of the approximation. There is no singularity in the algorithm proposed here. 
\end{enumerate}

\subsection{Comment on the Approximation}

Apart from the algorithms, we want to mention a few points related to an approximation for a surface.

\begin{enumerate}
	\item In Propositions \ref{prop: surface} and \ref{prop: approximation 2}, when the Hausdorff distance $\epsilon \rightarrow 0$, we haven't found a good way to express the target surface $S$ analytically, even though it is determined by the datum curve $\Gamma$, target curve $f_1$, and a folding angle $\rho_4$ (Proposition \ref{prop: surface}) or a sector angle $\alpha_1$ (Proposition \ref{prop: approximation 2}). We only give analytical expressions for special cases mentioned in Corollaries \ref{cor first special} and \ref{cor second special}. However, an analytical solution for how to approximate a curve is given in \cite{tachi_composite_2013}, which might be helpful in solving this problem.
	\item It is not necessary to make the rigid folding motion halt at the first column from the left. Such halting columns can be inserted to the creased paper arbitrarily. Besides, the two examples shown in sections 2.1 and 2.2 design the halting column differently, and there are many other possible techniques.
	\item The utilization of materials of our design depends on the proportion of halting columns with respected to the whole creased paper, which is at a relatively high level.
	\item If there are some holes on the target surfaces in Propositions \ref{prop: surface} and \ref{prop: approximation 2}, we can follow the algorithms described in this article and generate the approximation by applying kirigami on the creased papers.
	\item In this article we show that a series of developable surfaces can approximate a non-developable surface, which means the collection of all developable surfaces is not a closed set.
	\item Another possible design example is a creased paper with no inner panel (figure 4 in \cite{he_rigid_2020}), which can approximate more surfaces but will have some big ``cracks'' almost running through the final rigidly folded state. We will include the discussion on it in a future article.  
\end{enumerate}

\subsection{Optimization}

Optimizations can be applied to the algorithms mentioned in this article, which will lead to future work.

\begin{enumerate}
	\item In Propositions \ref{prop: approximation}, \ref{prop datum curve} and  \ref{prop: datum curve 2}, the piecewise-linear approximation do not need to be uniformly-spaced. Furthermore, the partition points are not necessarily on the curve. Given the number of partition points, there are some classic methods to approximate the datum and target curves, which will improve the accuracy of approximation.
	\item There are some other criteria for an approximation in Propositions \ref{prop: approximation}, \ref{prop datum curve} and \ref{prop: datum curve 2} under a given Hausdorff distance $\epsilon$, such as minimizing the number of inner vertices, the total bending energy \cite{solomon_flexible_2012}; or restricting the minimum length of all the creases, the minimum of sector angles, etc, which will involve more complex calculations.
\end{enumerate}

\section{Conclusion}

We have shown that it is possible to approximate some types of non-developable surfaces with a 1-DOF rigid folding motion starting from a planar creased paper. This might have useful engineering applications, for instance as a way of forming a shell structure in 3-dimensional space. However, given an arbitrary surface, the methods that can be used to generate an approximation with a planar creased paper that has limited DOFs are not fully understood.					

\section{Acknowledgment}

We thank Tom Aldridge for some preliminary works on this topic, and Hanxiao Cui for helpful discussions on Proposition \ref{prop: approximation}. This paper has been awarded the 7OSME Gabriella \& Paul Rosenbaum Foundation Travel Award.

\bibliographystyle{osmebibstyle}
\bibliography{Rigid-Folding}

\theaffiliations

\end{document}